\documentclass{article}
\usepackage[utf8]{inputenc}
\usepackage{amsmath}
\usepackage[T1]{fontenc}    % use 8-bit T1 fonts
\usepackage{hyperref}       % hyperlinks
\usepackage{url}            % simple URL typesetting
\usepackage{booktabs}       % professional-quality tables
\usepackage{amsfonts}       % blackboard math symbols
\usepackage{nicefrac}       % compact symbols for 1/2, etc.
\usepackage{microtype}      % microtypography
\usepackage{lipsum}
\usepackage{graphicx}
\usepackage[english]{babel}
\usepackage{amsthm}
\usepackage{tabularx}
\usepackage{array}

\graphicspath{ {./images/} }

\usepackage{mathtools}
\DeclarePairedDelimiter\ceil{\lceil}{\rceil}

\usepackage{amssymb}
\newcommand{\divides}{\mid}
\newcommand{\notdivides}{\nmid}
\newtheorem{theorem}{Theorem}[section]

\newtheorem{definition}[theorem]{Defintion}

\title{A New Operator for Egyptian Fractions}
\author{Keneth Adrian P. Dagal \\ 
  \texttt{kendee2012@gmail.com} \\}

\begin{document}

\maketitle

\begin{abstract}
 This paper introduces a new equation for rewriting two unit fractions to another two unit fractions. This equation is useful for optimizing the elements of an Egyptian Fraction. Parity of the elements of the Egyptian Fractions are also considered. And lastly, the statement that all rational numbers can be represented as Egyptian Fraction is re-established.  
\end{abstract}

\section{Introduction}
 
Unit fractions are fractions of the form $\frac{1}{n}$ for all integers $n > 1$. We define the set $U_f$ to be the collection of all unit fractions.
$$U_f = \left\{ \frac{1}{n} \mid n\in \mathbb{Z^+}-\{ 1 \}\right \}$$

The power set of $U_f$, denoted by $\mathcal{P}(U_f)$, is the set that contains all possible subset of $U_f$. We define the set $X$ be an arbitrary element of $\mathcal{P}(U_f)$ and we \it{partition} \normalfont the set
$\mathcal{P}(U_f)$ in terms of each set's cardinality ($\mid X\mid$) into three sets defined as:

$$\mathcal{P}(U_f) = \mathcal{P}(U_f)_{\mid X\mid < 2} \cup  \mathcal{P}(U_f)_{\mid X\mid \geq 2} \cup \mathcal{P}(U_f)_{\mid X\mid = \infty} $$

The set $\mathcal{P}(U_f)_{\mid X\mid < 2}$ contains the null set and the sets that contains each unit fraction. So, the set $\mathcal{P}(U_f)_{\mid X\mid < 2}$ is of no interest. For the set $\mathcal{P}(U_f)_{\mid X\mid = \infty}$, this is of great interest since our fundamental operation is to add all the elements of $X \in \mathcal{P}(U_f)_{\mid X\mid = \infty}$. So, we can classify all infinite series in the set $X \in \mathcal{P}(U_f)_{\mid X\mid = \infty}$ into either divergent or convergent. Before proceeding to some examples for the set $\mathcal{P}(U_f)_{\mid X\mid = \infty}$, we define some notations for simplicity.\\

For convenience, we define the function $S: X \rightarrow \mathbb{R}$ be $$S(X) = \sum_{x \in X }{x}.$$ And we define the set $N$ be $N= \{n \mid n= x^{-1} \text { for all } x\in X\}.$ Equivalently, $X = \{x \mid x= n^{-1} \text { for all } n\in N\}$. With this, we can redefine the function $S$ as $S:N \rightarrow \mathbb{R}.$ For example, we have $N = \mathbb{Z^+}-\{1\}$. Therefore, $S(N) = \infty$. This is known as the \textit{Harmonic Series} (without the term 1) which is known to be divergent.\\

For simplicity, we use the Riemann zeta function ($\zeta(s)$) and limit the domain of $s$ in $\mathbb{Z^+}$ to illustrate the function $S$ for some $X \in \mathcal{P}(U_f)_{\mid X\mid = \infty}.$ It is known that $\zeta(1)=\infty$, $\zeta(2)=\frac{\pi^2}{6} \approx 1.645$. For the set $N= \{n \mid n= q^2 \text{ for all integers } q\geq 2\}$, $S(N)= \frac{\pi^2}{6} -1$. Setting that aside, the set $\mathcal{P}(U_f)_{\mid X\mid \geq 2}$ is our major concern. 

\begin{definition}
The function $S$ is said to be an Egyptian fraction if $$S(X) = \sum_{x \in X }{x}.$$ for all $X \in \mathcal{P}(U_f)_{\mid X\mid \geq 2}$.
\end{definition}

It has already been established that \textit{Every positive rational number can be represented by an Egyptian Fraction}. We attempt to re-established it in the next section.

\section{The Inverse of the Function S}

In this section, we focus on the function $S$ where the domain is $\mathcal{P}(U_f)_{\mid X\mid \geq 2}$. In this domain, the function then becomes $S: X \rightarrow \mathbb{Q^+}$ where $\mathbb{Q^+}$ is the set of positive rational number

$$\mathbb{Q^+} = \left\{ \frac{a}{b} \mid a, b\in \mathbb{Z^+} \wedge (a,b)= 1 \right\}$$

The notation $(a,b)= 1$ means that $a$ and $b$ are relatively prime. We are interested in the inverse of the function $S$ ( which is not anymore a function) since we can have several $X$'s for a particular element in $\mathbb{Q^+}$. Our question is: 
\textit{Are all elements of $\mathbb{Q^+}$ defined for $S^{-1}$ ?}\\

To answer the question above,  we \it{partition} \normalfont the set $\mathbb{Q^+}$ into two subsets, namely:
 
 $$\mathbb{Q^{\geq}} = \left\{ \frac{a}{b} \mid  \frac{a}{b}\in \mathbb{Q^+} \wedge a \geq b \right\}$$
 
 $$\mathbb{Q^{<}} = \left\{ \frac{a}{b} \mid  \frac{a}{b}\in \mathbb{Q^+} \wedge a < b \right\}$$

One important known theorem is given below:

\begin{theorem}{Division Algorithm \cite{burton}}
 
 Given $a$ and $b$, with $b \neq 0$, there exists unique integers $q$ and $r$ such that $$a= bq+r$$ and $ 0 \leq r < \mid b \mid.$
\end{theorem}

With the previous theorem, we can focus on the set $\mathbb{Q^{<}}$ instead of $\mathbb{Q^+}$ since each element in $\mathbb{Q^{\geq}}$ can be written as $$\frac{a}{b} = q +\frac{r}{b}$$ such that $q$ is an integer and  $\frac{r}{b} \in \mathbb{Q^{<}} $. It is sufficient to show that each element in $\mathbb{Q^{\geq}}$ is defined for $S^{-1}$ by showing $ S^{-1}$ is defined for $\mathbb{Z^+}-\{1\}$ and $\mathbb{Q^{<}}$, and whenever $X_1 \cup X_2$ for all $X_1$ in $\mathbb{Z^+}-\{1\}$ and $X_2$ in $\mathbb{Q^{<}}$ ,  $X_1 \cap X_2= \emptyset$.

We start with $\frac{r}{b} \in \mathbb{Q^{<}} $.

The first splitting recursive equation is what we call the \textit{greedy algorithm for Egyptian fractions.}

\begin{theorem}{Greedy Algorithm}

Let $\frac{a}{b} \in \mathbb{Q^{<}}$,  $a=a_0$, $b=b_0$, $u_{i+1}=\ceil{b_i/a_i}$, $a_{i+1}= a_i\cdot u_{i+1}-b_i$, and $b_{i+1}=b_i \cdot u_{i+1}$
and the recurrence relation below:
$$\frac{a_i}{b_i} =\frac{1}{u_{i+1}} + \frac{a_{i}\cdot u_{i+1}-b_i}{b_i \cdot u_{i+1}}$$.

Initialize at $i=0$. While $a_i \nmid b_i$, add 1 to $i$, and use the recurrence relation until $a_i \mid b_i$. As such, the smallest fraction in the expansion of $\frac{a}{b}$ is $\frac{1}{u_n}$ with $i=n-2$. And the resulting expansion of $\frac{a}{b}$ is $\sum_{i=1}^n\frac{1}{u_i}.$

\end{theorem}

Some proofs of the theorem above is given in \cite{gc}.

\begin{theorem}

 The relation $S^{-1}: \mathbb{Q^{<}} \rightarrow X$ is well-defined.
 
\end{theorem}

\begin{proof}
 Let $c= \frac{a}{b} \in \mathbb{Q^{<}} $ and $1/x_i \in X$ for $i=1,2,3, \cdots, t-1, t$. Clearly, $t$ is the cardinality of set $X$. By greedy algorithm, existence of the finite set $X$ is immediate wherein the $n$'s are the $u_i$'s.
\end{proof}
 
\begin{theorem}{\cite{botts}}

Let $\mathbb{X}$ be the set that contains all $X$'s in the relation $S^{-1}: 1 \rightarrow X$. $\mid\mathbb{X}\mid =\infty$.
\end{theorem}

Botts(1967) had proven this theorem and explained in detail the structure of the denominators at each stage of the chain reaction.

\begin{theorem}
The relation $S^{-1}: \mathbb{Q^{\geq}} \rightarrow X$ is well-defined.
\end{theorem}

\begin{proof}
By theorem 2.1, Each $\frac{a}{b} \in \mathbb{Q^{\geq}}$ can be written as $$\frac{a}{b} = q +\frac{r}{b}$$ for integers $q$ and $r$. Thus by theorem 2.3 and 2.4, we have

$$\frac{a}{b} = \sum_{}^{}\frac{1}{l} + \sum_{}^{}\frac{1}{m}$$ for some collection of $l$ and $m$ in $\mathbb{Z^+} -\{1\}$. To guarantee that the expansion have all unique $l$'s and $m$'s, we start with the $m$'s. We know that by theorem 2.3, we have all unique $m$'s. All we need to do now is to guarantee that all $l$'s are not equal to any $m$, and each $l$ is unique in the sum. To do this, we let $\frac{1}{m_k}$ be the smallest term in the expansion of $\frac{r}{b}$. By theorem 2.4, we can start to expand $1$ at any starting point, $l_i\in \mathbb{Z^+} -\{1\}$ and we let this $l_i$ be equal to $ m_k+1$. And thus generate the expansion for 1. And since $q$ can be written as sum of 1's, we can redo the process by simply making the smallest term in the first expansion of 1, say $\frac{1}{l_k}$, and make a new expansion for the next 1 of the expansion of $q$ by starting at $\frac{1}{l_k+1}$. 
\end{proof}

 In conclusion, the answer to our question is: Yes, all elements of $\mathbb{Q^+}$  are defined for the relation $S^{-1}$.

\section{Operations for Egyptian Fractions}

The function $S$ is a many-to-one function that is why $S^{-1}$ is not a function. In this section, we focus on X's such that for $X_1$ and $X_2$ in 
$\mathcal{P}(U_f)_{\mid X\mid \geq 2}$, $S(X_1)= S(X_2)$. In this manner, we operate on the elements of $X$ to produce another $X'$ such that $S(X)= S(X')$ and use the notation $\mathbb{O}$ as the operator on X, $\mathbb{O}: X \rightarrow X'$.

\begin{definition}
Let $X$ be the original Egyptian Fraction and $X'$ be the new Egyptian Fraction from $X$ such that $S(X)= S(X')$. 
\begin{itemize}
    \item If $\mid X \mid  <  \mid X' \mid$, then the operator $\mathbb{O}$ is said to be a splitter, 
    \item If $\mid X \mid  =  \mid X' \mid$, then the operator $\mathbb{O}$ is said to be a rewriter, and 
    \item If $\mid X \mid  >  \mid X' \mid$, then the operator $\mathbb{O}$ is said to be a merger.
\end{itemize}

\end{definition}

We start with \textit{splitter} operator $\mathbb{O}$, $$\frac{1}{n}= \frac{1}{n+1} + \frac{1}{n(n+1)}$$
for $X$ such that $S(X)=1$.For simplicity of notation we can write the splitter operator ( or equation) above as $$(n) = (n+1, n(n+1)).$$

For example, 

$$1= \frac{1}{2}+\frac{1}{3}+\frac{1}{6}$$

$$1= \frac{1}{2}+\frac{1}{4}+\frac{1}{6}+\frac{1}{12}$$

 The first $N =\{2,3,6\}$ and $N'=\{2,4,6,12\}$. It can be seen that $N \cap N' = \{2,6\}$ which are the elements that did not change when operated. Evidently, 3 becomes 4 and 12 which came about using the splitter operator above with $n=3$.
 
 In fact, the aforementioned equation is only a special case of $$(ab) = (a(a+b) ,b(a+b))$$ when $a=1$ and $b=n$. And for odd denominators, we have $$1= \frac{1}{3}+\frac{1}{5}+\frac{1}{7}+\frac{1}{9}+\frac{1}{11}+\frac{1}{15}+\frac{1}{35}+\frac{1}{45}+\frac{1}{231} \cite{intef}.$$
 In Knott \cite{Knott}, Ben Thurston (May 2017) emailed Knott two other simple formulas::
 
 $$(ab) = (a(a+b) ,b(a+b))$$
 $$(abc) = (a(ab+bc+ac) ,b(ab+bc+ac), c(ab+bc+ac))$$
 
 Generally, it is easy to see that 
 
 $$(\prod_{i=1}^mx_i)= (x_1\cdot z,x_2 \cdot z, \cdots , x_{m-1}\cdot z, x_m \cdot z)$$ where $$z =\sum_{j=1}^m\frac{1}{x_j}(\prod_{i=1}^mx_i)$$

To illustrate,  we let $x_i=i+1$ for $i=1,2,3,4,5$. Thus $$(720) = (7200,4800,3600,2880,2400)$$
Simplifying the above equation, $$\frac{1}{3}= \frac{1}{10}+\frac{1}{12}+\frac{1}{15}+\frac{1}{20}+\frac{1}{30}.$$

The methods illustrated above are splitter operators from one term to at least two terms. The latest example $(3)= (10,12,15,20,30)$ is from one term to five terms. The next section includes  the parity condition on splitting the fraction into $2$ or $3$ parts (least number of parts needed for splitting).

\section{Splitter Operator with Parity Condition}

In the work of Knott \cite{Knott}, the case for splitting even to two even Egyptian fractions is given below which he called the Expand Even Rule,

$$(2n) = (2(n+1),2n(n+1))$$ for integers $n \geq 2$.

This came from the splitting equation $$(n) = (n+1, n(n+1))$$ multiplied both sides by $\frac{1}{2}$. The reason for such action is to expand by preserving evenness since for $$(n) = (n+1, n(n+1)).$$ Denote $e$ for even and $o$ for odd, we have all the possible cases below:

\begin{center}
    \begin{tabular}{|c|c|c|}
    \hline
        $n$ & $n+1$ & $n(n+1)$ \\ \hline
        e & o & e \\ \hline
        o & e & e \\ \hline
    \end{tabular}
\end{center}

Clearly, the splitting equation $(n) = (n+1, n(n+1))$ is not a \textit{parity-preserving splitting equation}.

\begin{definition}
An operator  $$(a_1, a_2,..., a_{n-1}, a_n) = (b_1, b_2,...,b_{m-1}, b_{m})$$ is said to be a parity-preserving operator ( or equation) if all of the entries are of the same parity.
\end{definition}

To illustrate the above definition, we take the splitter operator $(n) = (n+1, n(n+1))$ where  $a_1 = n$, $b_1 = n+1$, and $b_2= n(n+1)$.
 
 The equation $(2n) = (2(n+1),2n(n+1))$ is clearly a parity-preserving equation. In this equation,  all terms are even. For odd, we have the parity-preserving splitting equations below which are given by Peter in \cite{SplitEq}.
 
 If $k$ is odd,
 $$\frac{1}{2k+1}= \frac{1}{3k+2} + \frac{1}{6k+3}+\frac{1}{18k^2+21k+6}$$
 
 If $k$ is even, 
 $$\frac{1}{2k+1}= \frac{1}{3k+3} + \frac{1}{6k+3}+\frac{1}{6k^2+9k+3}$$
 
 We investigate the equations above by redefining and stating them formally:
 
 \begin{theorem}
 Let $n=2k+1$, $a=3$, $b= 3k+2$, and $c=k+1$ for positive integers $k$. If $k$ is odd, then $(n)= (b,an,abn)$.Otherwise for even $k$, $(n)= (b+1,an,a\cdot(b+1)\cdot(n))$. 
 \end{theorem}
 
 \begin{proof}
 We split $\frac{1}{n}$ to three equal parts: $(n)= (3n,3n,3n)$, so $a=3$.The middle part is $3n$, which is odd since $n=2k+1$ for positive integers $k$.Now, let $(n)= (b,3n,b \cdot 3n)$. Finding $b$, we have $$\frac{1}{b}+\frac{1}{b\cdot 3n}= \frac{2}{3n}$$ Solving the equation, $3n+1=2b$, then $$b = \frac{3(2k+1)+1}{2}= 3k+2$$ Therefore $b$ is odd only if $k$ is odd. So, if $k$ is even, then $b+1$ must be odd.
 \end{proof}
 
 A table below gives a summary of the previous theorem:
 \begin{center}
    \begin{tabular}{|c|c|c|c|c|c|c|}
    \hline
       $n$ & $k$ & $b$ & $b+1$ & $an$ & $abn$ & $a\cdot(b+1)\cdot n$\\ \hline
        o & o & o & e  & o & o & e\\ \hline
        o & e & e & o  & o & e & o\\ \hline
    \end{tabular}
\end{center}
 
 The equation above splits an odd $n$ to three odd terms because it is impossible to split odd $n$ into two parts. Also, the table illustrates that for odd $n$ where the $k$ in $n=2k+1$, If $k$ is odd, then  all entries in $(b,an,abn)$ are odd. And if $k$ is even, then all entries in $(b+1,an,a\cdot(b+1)\cdot(n))$ are odd.
 
\begin{theorem}
 There exists no odd positive integers $a$ and $b$ greater than 1 for odd number $n$ such that $(n)= (a,b)$. 
\end{theorem}

\begin{proof}
 Suppose there exists odd $a$ and $b$ for odd $n$ such that $(n)= (a,b)$.  Thus, $(n)= (2k_1+1,2k_2+1)$ for positive integers $k_1, k_2$. Since $n$ is an integer, then $2(k_1+k_2+1) \divides (2k_1+1)(2k_2+1)$. But this is false since $2\notdivides ab$.
\end{proof}
 
 \section{ A Rewriter for Egyptian Fractions}
 
 So far, what we have is to split an unit fraction to at least two parts. In this section, we introduce an equation that can rewrite two unit fractions into two another unit fractions which can be seen as an operator.
 
 \begin{theorem}
 Let $d$ and $q$ be positive integers where $q > 1$. If $r = q+d$, and $s = qr-d$, then $$(s)(rs)=(qr)(qs)$$
\end{theorem}

\begin{proof}
 
 \begin{equation} 
\begin{split}
\frac{1}{qr}+\frac{1}{qs} & = \frac{1}{r}\left(\frac{1}{q}+\frac{r}{qs}\right) \\
 & = \frac{1}{r}\left(\frac{1}{q}+\frac{d}{qs}+\frac{1}{s}\right)\\
 & = \frac{1}{r}\left(\frac{qr-d+d}{qs}+\frac{1}{s}\right)\\
 & = \frac{1}{r}\left(\frac{qr}{qs}+\frac{1}{s}\right)\\
 & = \frac{1}{r}\left(\frac{r}{s}+\frac{1}{s}\right)\\
 & = \frac{1}{s}+\frac{1}{rs}\\
\end{split}
\end{equation}
 \end{proof}

The terms in the equation above are $s$, $rs$ ,$qr$, $qs$. With these, we explore the related inequality for the terms.

\begin{theorem}
Let $d$ and $q$ be positive integers where $q > 1$, $r = q+d$, and $s = qr-d$, then $$q < r < s < qr < qs < rs.$$
\end{theorem}

\begin{proof}
First, $q< r$ is true, then we establish $r < s$. We start with the fact that $1 < 9$, then $(d^2+6d +1) < (d^2+6d+9)$.
And then, $$ (d^2+6d +1) < (d+3)^2$$
 $$ \sqrt{(d^2+6d +1)} < d+3$$
  $$ 1-d + \sqrt{(d^2+6d +1)} < 4$$
$$ \frac{1-d + \sqrt{(d^2+6d +1)}}{2} < 2$$

Note that $q$ is at least 2. Thus, $\frac{1-d + \sqrt{(d^2+6d +1)}}{2} < 2 \leq q$ And then,
$$ \frac{1-d + \sqrt{(d^2+6d +1)}}{2} < q$$
 $$ \frac{-(d-1) + \sqrt{(d-1)^2 - 4(-2d)}}{2} < q$$
 $$ 0 < q^2 + (d-1)q -2d$$
$$ 0 < (q-1)(q+d)-d$$
  $$0 < qr-r+d$$
 $$r < qr-d$$
$$r < s$$

With this,all other related inequalities are straightforward.
\end{proof}

To end this section, we show the parity table of the equation above.

\begin{center}
    \begin{tabular}{|c|c|c|c|c|c|c|}
    \hline
       $d$ & $q$ & $r$ & $s$ & $qr$ & $qs$ & $rs$\\ \hline
        o & o & e & o  & e & o & e\\ \hline
        o & e & o & o  & e & e & o\\ \hline
        e & o & o & o  & o & o & o\\ \hline
        e & e & e & e  & e & e & e\\ \hline
    \end{tabular}
\end{center}

The table generated in this paper follows the Boolean algebra such that $e=0$ and $o=1$. In addition, we have a special theorem below about odd parity.

\begin{theorem}
Let $r = q+d$, and $s = qr-d$.  The splitting equation $(s)(rs)=(qr)(qs)$ is an odd parity preserving equation if and only if the integer $q > 1$ is odd, and the value of $d$ is a positive even number.
\end{theorem}

\begin{proof}
We begin the proof by assuming the equation to be an odd parity preserving equation. As such, we define the function $p(t) = e$ if the expression $t$ is even, and $p(t) = o$ if $t$ is odd. Thus, $$p(s)=p(rs)=p(qr)=p(qs) = o.$$ Since $p(qr)$= o then, $p(q) = p(r)=o$. And since $p(r) = p(q+d) = o$, then $p(q)$ and $p(d)$ must have a different parity. But since, $p(q)$ is odd, then $p(d)$ must be even. As for the other direction, 
if $p(q)=o$ and $p(d)=e$, then $p(r) = o$, $p(s) = p(qr-d) =p(q)\cdot p(r)-p(d) =o\cdot o - e =o$. And since we know that $ o\cdot o = o$, then $p(rs)=p(qr)=p(qs) = o$.
\end{proof}

Ultimately, we generate five examples of the previous theorems

\begin{center}
    \begin{tabular}{|c|c|c|c|c|c|c|}
    \hline
       $d$ & $q$ & $r$ & $s$ & $qr$ & $qs$ & $rs$\\ \hline
        1 & 2 & 3 & 5  & 6  & 10 & 15\\ \hline
        2 & 3 & 5 & 13  & 15 & 39 & 65\\ \hline
        2 & 5 & 7 & 33  & 35 & 165 & 231\\ \hline
        4 & 3 & 7 & 17  & 21 & 51 & 119\\ \hline
        4 & 5 & 9 & 41  & 45 & 205 & 369\\ \hline
    \end{tabular}
\end{center}
 Rewriting some examples in the table in a conventional form, we have the following identities:
 
 $$\frac{1}{6}+\frac{1}{10}= \frac{1}{5}+\frac{1}{15}$$
 $$\frac{1}{15}+\frac{1}{39}= \frac{1}{13}+\frac{1}{65}$$
 $$\frac{1}{21}+\frac{1}{51}= \frac{1}{17}+\frac{1}{119}$$
 
\section{Acknowledgement}

The author would like to thank Peter and Ross Millikan for answering my questions and sharing useful materials at math.stackexchange.com. The author would like to thank Jose Arnaldo Dris for valuable conversation and unwavering support in preparing this manuscript.

\end{document}